\documentclass{ijnamb}
\def\currentvolume{1}
\def\currentissue{1}
\def\currentyear{2004}
\def\month{March}
\pagespan{1}{18}
\copyrightinfo{2004}{} 

\usepackage{elf}

\begin{document}
    \title[Comforming FE Discretization of the Unsteady QGEs]{A Conforming
    Finite Element Discretization of the Streamfunction Form of the Unsteady
    Quasi-Geostrophic Equations}

    \author{Erich L Foster}
    \address{Basque Center for Applied Mathematics, Alameda Mazarredo, 14,
      48009 Bilbao, Basque Country -- Spain}
    \email{efoster@bcamath.org}
    \urladdr{http://www.math.vt.edu/people/erichlf}

    \author{Traian Iliescu \and David Wells}
    \address{Department of Mathematics, Virginia Tech, Blacksburg, VA
      24061-0123, U.S.A.}
    \email{iliescu@vt.edu \and drwells@vt.edu}
    \urladdr{http://www.math.vt.edu/people/iliescu \and
    http://www.math.vt.edu/people/drwells}


    \abstract{This paper presents a conforming finite element
    semi-discretization of the streamfunction form of the one-layer unsteady
    quasi-geostrophic equations, which are a commonly used model for large-scale
    wind-driven ocean circulation. We derive optimal error estimates and present
    numerical results.}


    \maketitle


  \section{Introduction} \label{sec:Intro}
  The \emph{quasi-geostrophic equations (QGE)}, a standard simplified mathematical
model for large scale oceanic and atmospheric flows \cite{Cushman11, Majda,
Pedlosky92, Vallis06}, are often used in climate models \cite{Dijkstra05}. We
consider a finite element (FE) discretization of the QGE to allow for better
modeling of irregular geometries. Indeed, it is important to represent features
like coastlines in ocean models; numerical artifacts can result from stepwise
boundaries, which can affect ocean circulation predictions
over long time integration \cite{Adcroft98, Dupont03, Wang08}.

Most analyses of the QGE have been done on the mixed streamfunction-vorticity
rather than the pure streamfunction form. This work focuses on the
latter, which has the advantage of known optimal error estimates (see the error
estimate 13.5 and Table 13.1 in \cite{Gunzburger89}). However, the disadvantage
of not using a mixed formulation is that the pure streamfunction form of the
QGE is a fourth-order problem: this necessitates the use of a $C^1$ FE space for
a conforming FE discretization.

In what follows we first introduce, in \autoref{sec:Intro}, the
streamfunction-vorticity form of the QGE and its nondimensionalization, followed
by the pure streamfunction form of the QGE\@. In \autoref{sec:Disc} we introduce
the functional setting and the FE discretization in space. From there, we
develop optimal error estimates in \autoref{sec:Errors} followed by, in
\autoref{sec:Results}, numerical verification of the error estimates developed
in \autoref{sec:Errors}.

  \section{The Quasi-Geostrophic Equations} \label{sec:Eqs}
  The QGE are usually written as follows (\eg equation (14.57) in
\cite{Vallis06}, equation (1.1) in \cite{Majda}, equation (1.1) in
\cite{Wang94}, and equation (1) in \cite{Greatbatch00}):
\begin{align}
  \frac{\partial q}{\partial t} + J(\psi , q) &= A \, \Delta q + F
    \label{qge_q_psi_dim_1}                                                   \\
  q &= \Delta \psi + \beta \, y , \label{qge_q_psi_dim_2}
\end{align}
where $q$ is the potential vorticity, $\psi$ is the velocity streamfunction,
$\beta$ is the coefficient multiplying the $y$-coordinate (which is oriented
northward) in the $\beta$-plane approximation \eqref{eqn:beta_plane}, $F$ is the
forcing, $A$ is the eddy viscosity parameterization, and $J(\cdot , \cdot)$ is
the Jacobian operator given by
\begin{align}
  J(\psi , q) := \frac{\partial \psi}{\partial x} \, \frac{\partial q}{\partial y} -
    \frac{\partial \psi}{\partial y} \, \frac{\partial q}{\partial x} . \label{eqn:jacobian}
\end{align}
The $\beta$-plane approximation reads
\begin{equation}
  f = f_0 + \beta \, y , \label{eqn:beta_plane}
\end{equation}
where $f$ is the Coriolis parameter and $f_0$ is the reference Coriolis
parameter (see the discussion on page 84 in \cite{Cushman94} or Section
2.3.2 in \cite{Vallis06}).
As noted in Chapter 10.7.2 in \cite{Vallis06} (see also \cite{San11}), the eddy
viscosity parameter $A$ in \eqref{qge_q_psi_dim_1} is usually several orders of
magnitude higher than the molecular viscosity. This choice allows the use of a
coarse mesh in numerical simulations. The horizontal velocity $\mathbf{u}$ can
be recovered from $\psi$ and $q$ by the formula
\begin{align}
  \mathbf{u} := \nabla^{\perp} \psi =
    \begin{pmatrix} - \frac{\partial \psi}{\partial y}                        \\
    \frac{\partial \psi}{\partial x}
  \end{pmatrix} .
\label{eqn:u_psi}
\end{align}
The computational domain considered in this report is the standard
\cite{Greatbatch00} rectangular, closed basin on a $\beta$-plane with the
$y$-coordinate increasing northward and the $x$-coordinate eastward. The center
of the basin is at $y=0$, the northern and southern boundaries are at $y = \pm
\, L$, respectively, and the western and eastern boundaries are at $x = 0$ and
$x = L$ (see Figure 1 in \cite{Greatbatch00}).

We are now ready to nondimensionalize the QGE
\eqref{qge_q_psi_dim_1}-\eqref{qge_q_psi_dim_2}. There are several ways of
nondimensionalizing the QGE, based on different scalings and involving different
parameters (see standard textbooks on geophysical fluid dynamics, such as
\cite{Cushman11,Majda,Pedlosky92,Vallis06}). Since the FE error analysis in
this report is based on a precise relationship among the nondimensional
parameters of the QGE, 
we present a careful nondimensionalization of the QGE below. We first need to
choose a length scale and a velocity scale-- the length scale we choose is $L$,
the width of the computational domain. To define the velocity scale, we first
need to specify the forcing term \(F\) in \eqref{qge_q_psi_dim_1}. To this end,
we follow the presentation in Section 14.1.1 in \cite{Vallis06} and assume that
\(F\) is the wind-stress curl at the top of the ocean:
\begin{equation}
  \label{eqn:ForcingWindstress}
  F = \dfrac{1}{\rho} \left(\dfrac{\partial \tau^y}{\partial x} -
  \dfrac{\partial \tau^x}{\partial y}\right),
\end{equation}
where \(\rho\) is the density of the fluid, \(\boldsymbol{\tau} = (\tau^x, \tau^y)\)
is the wind-stress at the top of the ocean (see also Section 2.12 and equation
(14.3) in \cite{Vallis06} and Section 5.4 in \cite{Cushman94}) and is measured
in \(N/m^2\) (\eg page 1462 in \cite{Greatbatch00}). To determine the
characteristic velocity scale, we use the Sverdrup balance given in equation
(14.20) in \cite{Vallis06} (see also Section 8.3 in \cite{Cushman94}):
\begin{equation}
  \label{eqn:VelocityHeightAndForcing}
  \beta \int v dz = \dfrac{1}{\rho} \left(\dfrac{\partial \tau^y}{\partial x} -
  \dfrac{\partial \tau^x}{\partial y}\right),
\end{equation}
in which the velocity component \(v\) is integrated along the depth of the
fluid. The Sverdrup balance in \eqref{eqn:VelocityHeightAndForcing} represents
the balance between wind-stress (\ie forcing) and \(\beta\)-effect, which yields
the Sverdrup velocity
\begin{equation}
    \label{eqn:SverdrupVelocity}
    U := \dfrac{\tau_0}{\rho H \beta L},
\end{equation}
where \(\tau_0\) is the amplitude of the wind stress and \(H\) is the depth of
the fluid. It is easy to check that the Sverdrup velocity defined in
\eqref{eqn:SverdrupVelocity} has velocity units. We note that the same Sverdrup
velocity is used in equation (8-11) in \cite{Cushman94} and on page 1462 in
\cite{Greatbatch00} (the latter has an extra \(\pi\) factor due to the
particular wind forcing employed). The Sverdrup velocity
\eqref{eqn:SverdrupVelocity} will be used as the characteristic velocity scale
in the nondimensionalization. Once the length and velocity scales are chosen,
the variables in the QGE \eqref{qge_q_psi_dim_1}-\eqref{qge_q_psi_dim_2} can be
nondimensionalized as follows:
\begin{align}
  x^* = \frac{x}{L}, \quad
  y^* = \frac{y}{L}, \quad
  t^* = \frac{t}{L / U}, \quad
  q^* = \frac{q}{\beta \, L}, \quad
  \psi^* = \frac{\psi}{U \, L} ,
  \label{eqn:nondimensional_variables}
\end{align}
where a superscript $^*$ denotes a nondimensional variable. We denote
derivatives taken with respect to nondimensional coordinates by \(\Delta^*\) and
$J^*(\cdot, \cdot)$. Using \eqref{eqn:nondimensional_variables}, the
nondimensionalization of \eqref{qge_q_psi_dim_2} is
\begin{align}
  \beta \, L \, q^* = \frac{1}{L^2} \, \Delta^* (U \, L \, \psi^*) + \beta \, (L \, y^*) .
  \label{qge_q_psi_nondim_1}
\end{align}
Dividing \eqref{qge_q_psi_nondim_1} by $\beta \, L$, we get:
\begin{align}
  q^* = \left( \frac{U}{\beta \, L^2} \right) \, \Delta^* \psi^* + y^* .
  \label{qge_q_psi_nondim_2}
\end{align}
Defining the \emph{Rossby number $Ro$} as
\begin{align}
  Ro := \frac{U}{\beta \, L^2} , \label{eqn:rossby_number}
\end{align}
equation \eqref{qge_q_psi_nondim_2} becomes
\begin{align}
  q^* = Ro \, \Delta^* \psi^* + y^* .
  \label{qge_q_psi_nondim_3}
\end{align}
Then we nondimensionalize \eqref{qge_q_psi_dim_1}. We start with the left-hand
side:
\newcommand{\pp}[2]{\dfrac{\partial #1}{\partial #2}}
\begin{align}
  \pp{q}{t} &= (\beta U) \pp{q^*}{t^*}, \label{eqn:qge_q_psi_nondim_5}        \\
  J(\psi,q) &= \pp{\psi}{x} \pp{q}{y} - \pp{\psi}{y} \pp{q}{x}
             = U \pp{\psi^*}{x^*} \beta \pp{q^*}{y^*}
             - U \pp{\psi^*}{y^*} \beta \pp{q^*}{x^*}
             = (\beta U) J^*(\psi^*,q^*). \label{eqn:qge_q_psi_nondim_6}
\end{align}
Next, we nondimensionalize the right-hand side of \eqref{qge_q_psi_dim_1}. The
first term can be nondimensionalized as
\begin{equation}
  \label{eqn:qge_q_psi_nondim_7}
  A \Delta q
  = A \left(\frac{\partial^2 q}{\partial x^2}
  + \frac{\partial^2 q}{\partial y^2} \right)
  = A \left(\frac{1}{L^2} \frac{\partial^2}{\partial {x^*}^2} (\beta L q^*) +
      \frac{1}{L^2}\frac{\partial^2}{\partial {y^*}^2} (\beta L q^*)\right)
  = A \frac{\beta}{L} \Delta^* q^* .
\end{equation}
Thus, inserting \eqref{eqn:qge_q_psi_nondim_5}-\eqref{eqn:qge_q_psi_nondim_7} in
\eqref{qge_q_psi_dim_1}, we get
\begin{align}
  ( \beta \, U) \, \frac{\partial q^*}{\partial t^*} + (  \beta \, U ) \, J^*(\psi^*,q^*)
  &= A \, \frac{\beta}{L} \, \Delta^* q^* + F .
  \label{eqn:qge_q_psi_nondim_8}
\end{align}
Dividing by $\beta \, U$, we get:
\begin{align}
  \frac{\partial q^*}{\partial t^*} + J^*(\psi^*,q^*) &= \left( \frac{A}{U \, L} \right) \,
    \Delta^* q^* + \frac{F}{\beta \, U} .
\label{eqn:qge_q_psi_nondim_9}
\end{align}
Defining the \emph{Reynolds number $Re$} as
\begin{align}
  Re := \frac{U \, L}{A} ,
  \label{eqn:reynolds_number}
\end{align}
equation \eqref{eqn:qge_q_psi_nondim_9} becomes
\begin{align}
  \frac{\partial q^*}{\partial t^*} + J^*(\psi^*,q^*) &= Re^{-1} \, \Delta^* q^*
    + \frac{F}{\beta\, U} .
  \label{eqn:qge_q_psi_nondim_10}
\end{align}
The last term on the right-hand side of \eqref{eqn:qge_q_psi_nondim_10} has the
following units:
\begin{equation}
    \label{eqn:ForcingUnits}
    \left[\dfrac{F}{\beta U}\right]
    \overset{\eqref{eqn:ForcingWindstress}, \eqref{eqn:SverdrupVelocity}}{\sim}
    \left[\dfrac{
     \dfrac{1}{\rho} \left(\dfrac{\partial \tau^y}{\partial x}
     - \dfrac{\partial \tau^x}{\partial y}\right)
    }
    {
      \beta \dfrac{\tau_0}{\rho H \beta L}
    }
    \right],
\end{equation}
which, after an obvious simplifications, is nondimensional. Thus,
\eqref{eqn:ForcingUnits} clearly shows that the last term on the right-hand side
of \eqref{eqn:qge_q_psi_nondim_10} is nondimensional, so
\eqref{eqn:qge_q_psi_nondim_10} becomes
\begin{align}
  \frac{\partial q^*}{\partial t^*} + J^*(\psi^*,q^*)
  &= Re^{-1} \, \Delta^* q^* + F^*,
\label{eqn:qge_q_psi_nondim_12}
\end{align}
where \(F^* = F/(\beta U)\). Dropping the $^*$ superscript in
\eqref{eqn:qge_q_psi_nondim_12} and \eqref{qge_q_psi_nondim_2}, we obtain the
nondimensional \textit{vorticity-streamfunction form} of the \emph{one-layer
quasi-geostrophic equations}
\begin{align}
  \frac{\partial q}{\partial t} + J(\psi , q) &= Re^{-1} \, \Delta q + F \label{qge_q_psi_1}\\
  q &= Ro \, \Delta \psi + y, \label{qge_q_psi_2}
\end{align}
where $Re$ and $Ro$ are the Reynolds and Rossby numbers, respectively.

Substituting \eqref{qge_q_psi_2} in \eqref{qge_q_psi_1} and dividing by $Ro$, we
get the \textit{pure streamfunction form} of the \emph{one-layer
  quasi-geostrophic equations}
\begin{align}
  \frac{\partial \left[ \Delta \psi \right]}{\partial t} - Re^{-1} \, \Delta^2 \psi + J(\psi
    , \Delta \psi) + Ro^{-1} \, \psi_x = Ro^{-1} \, F. \label{qge_psi_1}
\end{align}
We note that the streamfunction-vorticity form has two unknowns ($q$ and
$\psi$), whereas the streamfunction form has only one unknown ($\psi$). The
streamfunction-vorticity form, however, is more popular than the streamfunction
form, since the former is a second-order partial differential equation, whereas
the latter is a fourth-order partial differential equation\@.

We also note that \eqref{qge_q_psi_1}-\eqref{qge_q_psi_2} and \eqref{qge_psi_1}
are similar in form to the 2D Navier Stokes Equations (NSE) written in both the
streamfunction-vorticity and streamfunction forms.
There are, however, several significant differences between the QGE and the 2D
NSE\@. First, we note that the term $y$ in \eqref{qge_q_psi_2} and the
corresponding term $\psi_x$ in \eqref{qge_psi_1}, which model the \emph{rotation
effects} in the QGE, do not have counterparts in the 2D NSE\@. Furthermore, the
Rossby number, $Ro$, in the QGE, which is a measure of the rotation effects,
does not appear in the 2D NSE\@.

To ensure the velocity and the streamfunction are related by \(\mathbf{u} =
(\psi_y, -\psi_x)\) (which is the relation used in
\cite{Gunzburger89}), we will consider the QGE \eqref{qge_psi_1} with $\psi$
replaced with $-\psi$:
\begin{align}
  -\frac{\partial \left[ \Delta \psi \right]}{\partial t}
    + Re^{-1} \, \Delta^2 \psi + J(\psi , \Delta \psi)
    - Ro^{-1} \, \frac{\partial \psi}{\partial x} = Ro^{-1} \, F .
    \label{eqn:QGE_psi}
\end{align}

We consider the boundary and initial conditions
\begin{equation}
  \psi = \frac{\partial \psi}{\partial \mathbf{n}} = 0 \text{ on }
  \partial \Omega \text{ and }
  \psi(0) = \psi_0,
  \label{eqn:QGEBCsICs}
\end{equation}
which were used in \cite{Gunzburger89} for the streamfunction form of the 2D NSE\@.

  \section{Finite Element Discretization} \label{sec:Disc}
  In this section we build the mathematical framework for the FE discretization of
the QGE\@. To this end, we consider the strong formulation of the QGE in pure
streamfunction form \eqref{eqn:QGE_psi}. The following functional spaces will be
used:
\begin{align}
  L^2(0, T;H^2_0(\Omega)) &:= \left\{ \psi(t,\mathbf{x}):[0, T] \to H^2_0(\Omega):
    \int_{0}^{T}\! \|\Delta \psi\| \, dt < \infty \right\}                    \\
  L^{\infty}(0, T;H_0^1(\Omega)) &:= \left\{\psi(t,\mathbf{x}):[0, T] \to H_0^1(\Omega):
    \ess \sup_{0<t<T} \|\nabla \psi\| < \infty\right\}.
\end{align}
Additionally, let
\begin{equation}
  X := H^2_0(\Omega) = \left\{ \psi\in H^2(\Omega):
  \psi=\frac{\partial\psi}{\partial \mathbf{n}}=0
    \text{ on } \partial\Omega \right\}.
\end{equation}
Denote the \(L^2\) inner product by $(\cdot,\cdot)$. The strong formulation of
the QGE in pure streamfunction form \eqref{eqn:QGE_psi} reads: Find \(\psi \in
L^2(0, T; H_0^2(\Omega)) \cap L^\infty(0, T; H_0^1(\Omega))\) such that
\begin{align}
    (\nabla \psi_t, \nabla \chi) + Re^{-1} (\Delta
      \psi, \Delta \chi) + b(\psi;\psi,\chi) - Ro^{-1}(\psi_x,\chi)
      &= Ro^{-1} (F,\chi),\quad \forall \chi \in X,
  \label{eqn:QGEStrong}                                                       \\
  \psi(0) = \psi_0,
  \label{eqn:QGEStrongInitial}
\end{align}
where the trilinear form is defined as follows (see (13) in \cite{Foster} and
Section 13.1 in \cite{Gunzburger89}):
\begin{equation}
  b(\xi; \psi, \chi) = \int_{\Omega}\! \Delta \xi\, (\psi_y \chi_x - \psi_x
  \chi_y)\, d\mathbf{x}.
  \label{eqn:b}
\end{equation}
We assume that the strong formulation of the QGE
\eqref{eqn:QGEStrong}-\eqref{eqn:QGEStrongInitial} has a unique solution which
satisfies the following regularity property:
\begin{equation}
    \int_0^T \|\Delta \psi\|^4 dt < \infty.
\end{equation}
We note the solution of the strong formulation of the NSE satisfies a
similar regularity property (see definition 33 in \cite{Layton08}). We also
assume that \(\|F\|_{-2}\) is in \(L^2(0, T)\), where the dual norms are defined
by (see definition 24 in \cite{Layton08})
\begin{equation}
    \|F\|_{-1} = \sup_{v \in H_0^1(\Omega)} \dfrac{(F, v)}{|v|_1} \text{ and }
    \|F\|_{-2} = \sup_{v \in H_0^2(\Omega)} \dfrac{(F, v)}{|v|_2}.
\end{equation}

In what follows, we will use the following norms and seminorms: for all \(v \in
H_0^2(\Omega)\), we define (see \cite{Ciarlet}, page 14)
\begin{align*}
  |v|_2^2 &= \int_\Omega \left[
  \sum_{i = 1}^3 \left(\dfrac{\partial^2 v}{\partial x_i^2}\right)^2
  + \sum_{i, j = 1; i \neq j}^3
  \left(\dfrac{\partial^2 v}{\partial x_i \partial x_j}\right)^2
  \right] d\mathbf{x}                                                         \\
  ||\Delta v||^2 &= \int_\Omega \left[
  \sum_{i = 1}^3 \left(\dfrac{\partial^2 v}{\partial x_i^2}\right)^2
  + \sum_{i, j = 1; i \neq j}^3
  \dfrac{\partial^2 v}{\partial x_i^2} \dfrac{\partial^2 v}{\partial x_j^2}
  \right] d\mathbf{x}.                                                        \\
\end{align*}
It can be proven that \(|v|_2 = \|\Delta v\|,  \forall v \in X\), see (1.2.8) in
\cite{Ciarlet}. Thus, the seminorm \(v \to \|\Delta v\|\) is a norm in \(X =
H_0^2(\Omega)\), which is equivalent to the norm \(\|\cdot\|_2\). As a
byproduct, we obtain the following Poincar\'e-Friedrichs inequality: there
exists a finite, positive constant \(\Gamma_0\) such that for any \(\psi \in
H_0^2(\Omega)\),
\begin{equation}
  \label{eqn:pf-2nd-derivative}
  \|\nabla \psi\| \leq \Gamma_0 \|\Delta \psi\|.
\end{equation}

Let $\mathcal{T}^h$ denote a triangulation of $\Omega$ with mesh
size (maximum triangle diameter) $h$. We consider a \emph{conforming} FE
discretization of \eqref{eqn:QGEStrong}-\eqref{eqn:QGEStrongInitial}, \ie let
$X^h$ be piecewise polynomials such that $X^h \subset X = H_0^2(\Omega)$. The FE
discretization of the streamfunction form of the QGE
\eqref{eqn:QGEStrong}-\eqref{eqn:QGEStrongInitial} reads: Find \(\psi^h \in
L^2(0, T; X^h) \cap L^\infty(0, T; H_0^1(\Omega))\) such that, \(\forall \chi^h
\in X^h\),
\begin{align}
  \label{eqn:SemiDiscretization}
  (\nabla \psi^h_t, \nabla \chi^h)
  + Re^{-1} (\Delta \psi^h, \Delta \chi^h) + b(\psi^h;\psi^h,\chi^h) -
    Ro^{-1}(\psi^h_x,\chi^h) &= Ro^{-1} (F,\chi^h),                           \\
  \label{eqn:initialConditionProjection}
  \psi^h(0) &= \psi_0^h,
\end{align}
where \(\psi_0^h\) is the FE initial condition.
We assume \eqref{eqn:SemiDiscretization}-\eqref{eqn:initialConditionProjection}
has a unique solution \(\psi^h\).

  \section{Error Analysis} \label{sec:Errors}
  In this section we present the convergence and error analysis associated with
\eqref{eqn:SemiDiscretization}-\eqref{eqn:initialConditionProjection}. We
will use the same approach as the one used
in Section 4 of \cite{Foster}, which contains the error analysis for the
stationary QGE\@.

The following lemma will introduce some useful bounds for the forms introduced
in \autoref{sec:Disc}.
\begin{lemma} \label{lma:ContinuousForms}
  There exist finite constants $\Gamma_1, \Gamma_2 > 0$ such that for
  all $\psi, \chi, \varphi\in X$ the following inequalities hold:
  \begin{align}
    (\Delta \psi, \Delta \chi) &\le |\psi|_2\, |\chi|_2, \label{eqn:a1Cont}   \\
    b(\psi;\varphi,\chi) &\le \Gamma_1 |\psi|_2\, |\varphi|_2\, |\chi|_2,
      \label{eqn:BH2Bounds}                                                   \\
    (\psi_x,\chi) &\le \Gamma_2 \, |\psi|_2 \, |\chi|_2, \label{eqn:a3Cont}   \\
    (F,\chi) &\le \|F\|_{-2} \, |\chi|_2.
      \label{eqn:lCont}
  \end{align}
\end{lemma}
For a proof of this result, see (12)-(21) of \cite{Foster}, (5.7)-(5.10) of
\cite{FosterThesis}, and inequalities (2.2)-(2.3) in \cite{Cayco86}.

\begin{prop} \label{prop:Stability}
  The solution of
  \eqref{eqn:SemiDiscretization}-\eqref{eqn:initialConditionProjection},
  $\psi^h$, is stable; for any $t>0$ the following inequality holds:
  \begin{equation}
    \frac{1}{2}\|\nabla \psi^h(t)\|^2 + \frac{Re^{-1}}{2}\int_{0}^{t}\! \|\Delta
      \psi^h(t')\|^2 \, dt' \le \frac{1}{2} \|\nabla \psi^h_0\|^2
      + \frac{Re\, Ro^{-2}}{2} \int_{0}^{t}\! \|F(t')\|^2_{-2}\, dt'.
    \label{eqn:Stability}
  \end{equation}
\end{prop}
\begin{proof}
  Take $\chi^h = \psi^h$ in \eqref{eqn:SemiDiscretization} and note that
  $b(\psi^h;\psi^h,\psi^h) = 0$ and $(\psi^h_x, \psi^h) = 0$ (see Remark 1 in
  \cite{Foster}). Using the definition of the \(\|\cdot\|_{-2}\) norm we get
  \begin{equation}
    \label{eqn:qgeCancelation}
    \frac{1}{2} \frac{d}{dt} \|\nabla \psi^h\|^2 + Re^{-1} \|\Delta \psi^h\|^2 =
      Ro^{-1} (F,\psi^h) \le Ro^{-1} \|F\|_{-2}\,\|\Delta \psi^h\|.
  \end{equation}
  Using the Young inequality in \eqref{eqn:qgeCancelation} we have
  \begin{equation}
    \frac{1}{2} \frac{d}{dt} \|\nabla \psi^h\|^2 + Re^{-1} \|\Delta \psi^h\|^2
      \le \frac{Ro^{-2}}{2\epsilon} \|F\|_{-2}^2 + \frac{\epsilon}{2}\|\Delta
      \psi^h\|^2.
      \label{eqn:HolderStability}
  \end{equation}
  Taking $\epsilon = Re^{-1}$ in \eqref{eqn:HolderStability} results in
  \begin{equation}
    \label{eqn:YoungSubstitution}
    \frac{1}{2} \frac{d}{dt} \|\nabla \psi^h\|^2 + \frac{Re^{-1}}{2} \|\Delta
      \psi^h\|^2 \le \frac{Re\,Ro^{-2}}{2} \|F\|_{-2}^2.
  \end{equation}
  Since \(\|F\|_{-2} \in L^2(0, T)\), integrating
  \eqref{eqn:YoungSubstitution} over $(0,t)$ gives the final result.
\end{proof}

The following lemma will be used in the proof of \autoref{lma:BH1Bound}.
\begin{lemma} \label{lma:trilinear}
  For $\psi,\,\xi,\,\chi\in H^2_0(\Omega)$, we have
  \begin{equation}
    b(\psi; \xi, \chi) = b^*(\chi; \xi, \psi) - b^*(\xi; \chi, \psi),
    \label{eqn:trilinearSplitBStar}
  \end{equation}
  where
  \begin{equation}
      b^*(\xi, \psi, \phi) = \int_\Omega
      (\xi_y \psi_{xy} - \xi_x \psi_{yy}) \phi_y -
      (\xi_x \psi_{xy} - \xi_y \psi_{xx}) \phi_x d\mathbf{x}
    \label{eqn:trilinearDefineBStar}
  \end{equation}
\end{lemma}
For a proof, see equation (8) and Lemma 5.6 in \cite{Fairag98}.

\begin{lemma} \label{lma:BH1Bound}
  There exist finite constants $\Gamma_3,\Gamma_4>0$ such that, for all
  $\psi,\, \varphi,\, \chi \in X$, the following inequalities hold:
  \begin{align}
    b(\psi;\varphi,\chi) &\le \Gamma_3 \|\Delta \psi\|\, \|\Delta \varphi\|\,
      \left(\|\nabla \chi\|^{\nicefrac{1}{2}}
      \|\Delta \chi\|^{\nicefrac{1}{2}}\right) \label{eqn:BH1BoundChi}        \\
    b(\psi;\varphi,\chi) &\le \Gamma_4 \left(\|\nabla \psi\|^{\nicefrac{1}{2}}
      \|\Delta \psi\|^{\nicefrac{1}{2}}\right)\,
      \|\Delta \varphi\|\, \|\Delta \chi\|. \label{eqn:BH1BoundPsi}
  \end{align}
\end{lemma}
\begin{proof}
  To prove estimate \eqref{eqn:BH1BoundChi}, we apply the H\"older inequality to
  $b(\psi; \varphi, \chi)$:
  \begin{equation}
    \label{eqn:HolderInitialTrilinear}
    b(\psi;\varphi,\chi) \le \|\Delta \psi\|_{L^p} \|\nabla \varphi\|_{L^q}
    \|\nabla \chi\|_{L^r},
    \text{ where } \dfrac{1}{p} + \dfrac{1}{q} + \dfrac{1}{q} = 1.
  \end{equation}
  Letting \(p = 2\) and \(q = r = 4\) in \eqref{eqn:HolderInitialTrilinear}
  yields
  \begin{equation}
    \label{eqn:TrilinearHolderSubstitution}
    b(\psi;\varphi,\chi) \le \|\Delta \psi\| \|\nabla \varphi\|_{L^4} \|\nabla
      \chi\|_{L^4}.
  \end{equation}
  Applying the Ladyzhenskaya inequality twice (Theorem 4 in \cite{Layton08}) to
  the last two factors on the right hand side of \eqref{eqn:TrilinearHolderSubstitution}
  yields
  \begin{equation}
    \label{eqn:TrilinearApplyLadyzhenskaya}
    b(\psi;\varphi,\chi) \le \Gamma_5 \|\Delta \psi\|
      \|\nabla \varphi\|^{\nicefrac{1}{2}} \|\Delta \varphi\|^{\nicefrac{1}{2}}
      \|\nabla \chi\|^{\nicefrac{1}{2}} \|\Delta \chi\|^{\nicefrac{1}{2}},
  \end{equation}
  where \(\Gamma_5\) is a positive constant. Using \eqref{eqn:pf-2nd-derivative}
  on \(\|\nabla \varphi\|^{\nicefrac{1}{2}}\) in
  \eqref{eqn:TrilinearApplyLadyzhenskaya} gives
  \begin{equation*}
    b(\psi;\varphi,\chi) \le \Gamma_3 \|\Delta \psi\|\, \|\Delta \varphi\|\,
      \left(\|\nabla \chi\|^{\nicefrac{1}{2}}
      \|\Delta \chi\|^{\nicefrac{1}{2}}\right),
  \end{equation*}
  where $\Gamma_3$ is also a positive constant, which proves estimate
  \eqref{eqn:BH1BoundChi}.

  To prove estimate \eqref{eqn:BH1BoundPsi}, we first rewrite \(b(\psi; \varphi,
  \chi)\) with relations \eqref{eqn:trilinearSplitBStar} and
  \eqref{eqn:trilinearDefineBStar} in \autoref{lma:trilinear}:
  \begin{equation}
      \label{eqn:TrilinearFairag}
      b(\psi; \varphi, \chi) = b^*(\chi, \varphi, \psi)
      - b^*(\varphi, \chi, \psi).
  \end{equation}
  Next we apply the H\"older inequality to each of the terms on the right hand side of
  \eqref{eqn:TrilinearFairag}, obtaining
  \begin{equation}
      \label{eqn:TrilinearFairagHolder}
      b(\psi; \varphi, \chi) \leq
      \|\Delta \chi\| \|\nabla \varphi\|_{L^4} \|\nabla \psi\|_{L^4} +
      \|\Delta \varphi\| \|\nabla \chi\|_{L^4} \|\nabla \psi\|_{L^4}.
  \end{equation}
  We apply the Ladyzhenskaya inequality to each term on the right hand side of
  \eqref{eqn:TrilinearFairagHolder}:
  \begin{equation}
      \label{eqn:TrilinearFairagLadyzhenskaya}
      \begin{aligned}
        b(\psi; \varphi, \chi)\leq \,
        &\Gamma_6 \|\Delta \chi\|
        (\|\nabla \varphi\|^\half \|\Delta \varphi\|^\half)
        (\|\nabla \psi\|^\half \|\Delta \psi\|^\half)                         \\
        + \,
        &\Gamma_7 \|\Delta \varphi\|
        (\|\nabla \chi\|^\half \|\Delta \chi\|^\half)
        (\|\nabla \psi\|^\half \|\Delta \psi\|^\half),
      \end{aligned}
  \end{equation}
  where \(\Gamma_6\) and \(\Gamma_7\) are two positive constants. Finally, by
  applying \eqref{eqn:pf-2nd-derivative} to each term on the right hand side of
  \eqref{eqn:TrilinearFairagLadyzhenskaya} we achieve the desired result:
  \begin{equation*}
    b(\psi;\varphi,\chi) \le \Gamma_4 \left(\|\nabla \psi\|^{\nicefrac{1}{2}}
      \|\Delta \psi\|^{\nicefrac{1}{2}}\right)\,
      \|\Delta \varphi\|\, \|\Delta \chi\|.
  \end{equation*}
\end{proof}

The next theorem proves the convergence of the FE approximation $\psi^h$ to the
exact solution $\psi$. The proof is similar to
the proof for Theorem 22 in \cite{Layton08}.                                  \\
\begin{thm} \label{thm:SemiConvergence}
  Let $\psi$ be the unique solution of the QGE
  \eqref{eqn:QGEStrong}-\eqref{eqn:QGEStrongInitial} and \(\psi^h\) be its
  FE approximation in
  \eqref{eqn:SemiDiscretization}-\eqref{eqn:initialConditionProjection}. Then
  the following estimate holds:
  \begin{equation}
    \begin{split}
      \|\nabla \left( \psi - \psi^h\right)(T) \|^2
        &+ Re^{-1} \int_0^T\! \|\Delta \left(\psi - \psi^h\right)\|^2\, dt
        \le C\,\biggl\{\bigl\|\nabla\left(\psi - \psi^h\right)(0)\bigr\|^2    \\
      & + \inf_{\lambda^h : [0, T] \to X^h} \biggl[
        \|\nabla (\psi - \lambda_h)(0)\|
        + \int_0^T\! \bigl\|\nabla \left(\psi - \lambda^h\right)_t\bigr\|^2
        + \bigl\|\Delta \left(\psi - \lambda^h\right)\bigr\|^2\, dt           \\
      & + \bigl\|\Delta \left(\psi - \lambda^h\right)\bigr\|^2_{L^4(0,T;L^2)}
        + \|\nabla \left(\psi - \lambda^h\right)(T)\|^2\biggr]\biggr\},
    \end{split}
    \label{eqn:SemiConvergence}
  \end{equation}
  where $C$ is a generic positive constant which can depend on \(T, F, \psi_0,
  Re, Ro, \Gamma_0, \Gamma_1, \Gamma_2, \Gamma_3\), and \(\Gamma_4\), but not on
  the mesh size \(h\).
\end{thm}
\begin{proof}
  Let $\chi = \chi^h \in X^h$ and subtract \eqref{eqn:SemiDiscretization} from
  \eqref{eqn:QGEStrong}. Denoting $e:=\psi - \psi^h$, we obtain
  \begin{equation}
    \label{eq:errorInitial}
    (\nabla e_t, \nabla \chi^h) + \left[b(\psi;\psi,\chi^h) -
    b(\psi^h;\psi^h,\chi^h)\right]
    + Re^{-1}(\Delta e, \Delta \chi^h) - Ro^{-1} (e_x, \chi^h) = 0\quad
    \forall \chi^h \in X^h \subset X.
  \end{equation}
  Now adding and subtracting $b(\psi^h;\psi,\chi^h)$ in \eqref{eq:errorInitial}
  gives
  \begin{equation}
    \label{eq:errorPutTakeB}
    (\nabla e_t, \nabla \chi^h) + \left[b(e;\psi,\chi^h) +
    b(\psi^h;e,\chi^h)\right]
      + Re^{-1}(\Delta e, \Delta \chi^h) - Ro^{-1} (e_x, \chi^h) = 0\quad
      \forall \chi^h \in X^h \subset X.
  \end{equation}
  Taking $\lambda^h:[0,T] \to X^h$ arbitrary and decomposing the error in
  \eqref{eq:errorPutTakeB} as $e = \eta - \Phi^h$, where $\eta := \psi -
  \lambda^h$ and $\Phi^h := \psi^h - \lambda^h$, results in
  \begin{equation}
    \label{eqn:ErrorDecomposition}
    \begin{aligned}
      (\nabla \Phi^h_t, \nabla \chi^h) + Re^{-1}(\Delta \Phi^h, \Delta \chi^h)
        & = (\nabla \eta_t, \nabla \chi^h) + Re^{-1}(\Delta \eta, \Delta \chi^h)
        - Ro^{-1} \left[(\eta_x, \chi^h) - (\Phi^h_x, \chi^h)\right]          \\
      & + \left[ b(\eta;\psi,\chi^h) - b(\Phi^h;\psi,\chi^h)
        + b(\psi^h;\eta,\chi^h) - b(\psi^h;\Phi^h,\chi^h)\right].
    \end{aligned}
  \end{equation}
  Let $\chi^h = \Phi^h$ in \eqref{eqn:ErrorDecomposition}. Noting that
  $b(\psi^h;\Phi^h,\Phi^h) = 0$ and $(\Phi^h_x,\Phi^h) = 0$ (see Remark 1 in
  \cite{Foster}), we get
  \begin{equation}
    \begin{aligned}
      \frac{1}{2} \frac{d}{dt} \|\nabla \Phi^h\|^2 + Re^{-1}\|\Delta \Phi^h\|^2
         = (\nabla \eta_t, \nabla \Phi^h) &+ Re^{-1}(\Delta \eta, \Delta \Phi^h)
        - Ro^{-1} (\eta_x, \Phi^h)                                            \\
      & + \left[ b(\eta;\psi,\Phi^h) - b(\Phi^h;\psi,\Phi^h)
        + b(\psi^h;\eta,\Phi^h)\right].
    \end{aligned}
  \end{equation}
  Using definition 19 in \cite{Layton08}, the Cauchy-Schwarz inequality,
  \eqref{eqn:pf-2nd-derivative}, and \eqref{eqn:a3Cont} from
  \autoref{lma:ContinuousForms} we have
  \begin{equation}
    \begin{split}
      \frac{1}{2} \frac{d}{dt} \|\nabla \Phi^h\|^2 + Re^{-1}\|\Delta \Phi^h\|^2
        &\le \Gamma_0 \|\nabla \eta_t\| \|\Delta \Phi^h\|
        + Re^{-1}\|\Delta \eta\|\, \|\Delta \Phi^h\|
        + Ro^{-1} \Gamma_2 \|\Delta \eta\| \|\Delta \Phi^h\|                  \\
      & + \left[ b(\eta;\psi,\Phi^h) - b(\Phi^h;\psi,\Phi^h)
        + b(\psi^h;\eta,\Phi^h)\right].
    \end{split}
    \label{eqn:HolderError}
  \end{equation}
  Using the Young inequality with some $\epsilon > 0$ and estimate
  \eqref{eqn:BH2Bounds} from \autoref{lma:ContinuousForms}, we get
  \begin{align}
      \Gamma_0 \|\nabla \eta_t\| \|\Delta \Phi^h\|
        &\leq \dfrac{\epsilon}{2} \|\Delta \Phi^h\|^2
        + \frac{\Gamma_0^2}{2 \epsilon} \|\nabla \eta_t\|^2
        \label{eqn:YoungT}                                                    \\
    Re^{-1} \|\Delta \eta\| \|\Delta \Phi^h\|
      &\le \frac{\epsilon}{2} \|\Delta \Phi^h\|^2
      + \frac{Re^{-2}}{2 \epsilon} \|\Delta \eta\|^2 \label{eqn:YoungLaplace} \\
    Ro^{-1} \Gamma_2 \|\Delta \eta\| \|\Delta \Phi^h\|
      &\le \frac{\epsilon}{2} \|\Delta \Phi^h\|^2
      + \frac{Ro^{-2} \Gamma_2^2}{2 \epsilon} \|\Delta \eta\|^2.
      \label{eqn:YoungBeta}
  \end{align}
  Using the Young inequality with $\varepsilon > 0$ and
  estimate \eqref{eqn:BH2Bounds} in \autoref{lma:ContinuousForms} yields
  \begin{equation}
    b(\eta;\psi,\Phi^h)
    \le \Gamma_1 \|\Delta \eta\|\,\|\Delta \psi\|\, \|\Delta \Phi^h\|
    \le \frac{\varepsilon}{2} \|\Delta \Phi^h\|^2
      + \frac{\Gamma_1^2}{2 \varepsilon} \|\Delta \eta\|^2 \|\Delta \psi\|^2.
    \label{eqn:YoungTrilinear}
  \end{equation}
  Substituting $\varepsilon = 2 \epsilon$ in \eqref{eqn:YoungTrilinear} we
  obtain
  \begin{equation}
    b(\eta; \psi, \Phi^h) \le \epsilon \|\Delta \Phi^h\|^2
      + \frac{\Gamma_1^2}{4 \epsilon} \|\Delta \eta\|^2 \|\Delta \psi\|^2.
      \label{eqn:YoungBH2}
  \end{equation}
  Using \eqref{eqn:YoungT} - \eqref{eqn:YoungBH2} in \eqref{eqn:HolderError}
  we obtain
  \begin{equation}
    \begin{split}
    \frac{1}{2} \frac{d}{dt} \|\nabla \Phi^h\|^2 + \frac{1}{2}\left(2Re^{-1} -
      5 \epsilon \right)\|\Delta \Phi^h\|^2
      &\le \frac{1}{2 \epsilon}\left[\Gamma_0^2 \|\nabla \eta_t\|^2
      + \left( Re^{-2} + Ro^{-2} \Gamma_2^2 \right) \|\Delta \eta\|^2\right]  \\
     & + \frac{\Gamma_1^2}{4 \epsilon}\|\Delta \eta\|^2 \|\Delta \psi\|^2  -
     \left[b(\Phi^h;\psi,\Phi^h) - b(\psi^h;\eta,\Phi^h)\right].
    \end{split}
    \label{eqn:B1Inequality}
  \end{equation}
  For the term $b(\Phi^h; \psi, \Phi^h)$ we use \autoref{lma:BH1Bound} and the
  following version of the Young inequality (equation (1.1.4) in \cite{Layton08}):
  given $a,\,b>0$, for any $\epsilon > 0$ and pair \(p, q\) satisfying
  \begin{equation*}
    1\le p, q \le \infty, \quad \frac{1}{p} + \frac{1}{q} = 1
  \end{equation*}
  it holds that
  \begin{equation}
    \label{eqn:GeneralizedYoung}
    ab \le \epsilon\, a^p
    + \dfrac{\left(p\,\epsilon\right)^{-\nicefrac{q}{p}}}{q} b^q.
  \end{equation}
  Picking $p=\nicefrac{4}{3}$ and $q = 4$ in \eqref{eqn:GeneralizedYoung}, we
  obtain
  \begin{equation}
    \label{eqn:EpsYoungH1}
    |b(\Phi^h; \psi, \Phi^h)| \le \Gamma_3\, \|\Delta \Phi^h\|^{\nicefrac{3}{2}}
      \left(\|\Delta \psi\| \|\nabla \Phi^h\|^{\nicefrac{1}{2}}\right)
    \le \epsilon \|\Delta \Phi^h\|^2 + C^*_1(\Gamma_3,\epsilon) \|\Delta \psi\|^4
      \|\nabla \Phi^h\|^2,
  \end{equation}
  where $C^*_1(\Gamma_3,\epsilon) = \nicefrac{27}{256}\,\Gamma_3^4\,\epsilon^{-3}$.
  Combining \eqref{eqn:B1Inequality} and \eqref{eqn:EpsYoungH1} yields
  \begin{equation}
    \begin{split}
      \frac{1}{2} \frac{d}{dt} \|\nabla \Phi^h\|^2 + \frac{1}{2}\left(2Re^{-1} -
        7 \epsilon \right)
        &\|\Delta \Phi^h\|^2 \le
        \frac{1}{2 \epsilon}\left[\Gamma_0^2 \|\nabla \eta_t\|^2
        + \left( Re^{-2} + Ro^{-2} \Gamma_2^2 \right) \|\Delta \eta\|^2\right]\\
      & + \frac{\Gamma_1^2}{4\epsilon}\|\Delta \eta\|^2 \|\Delta \psi\|^2
        + C^*_1(\Gamma_3,\epsilon) \|\Delta \psi\|^4 \|\nabla \Phi^h\|^2
        + b(\psi^h;\eta,\Phi^h).
    \end{split}
    \label{eqn:B2Inequality}
  \end{equation}
  For the final term $b(\psi^h; \eta, \Phi^h)$, we use inequality
  \eqref{eqn:BH1BoundPsi} and the Young inequality with $\varepsilon = 2
  \epsilon$, \ie
  \begin{equation}
    \label{eqn:YoungPhih}
    b(\psi^h; \eta, \Phi^h) \le \Gamma_4\left(\|\nabla \psi^h\|^{\nicefrac{1}{2}}
      \|\Delta \psi^h\|^{\nicefrac{1}{2}}\right) \|\Delta \eta\|\,
      \|\Delta \Phi^h\|
    \le \epsilon \|\Delta \Phi^h\|^2 + \frac{\Gamma_4^2}{4\epsilon}
      \|\nabla \psi^h\|\, \|\Delta \psi^h\|\, \|\Delta \eta\|^2.
  \end{equation}
  By stability estimate \eqref{eqn:Stability} in \autoref{prop:Stability}, we have
  \begin{equation}
    \|\nabla \psi^h\| \le C^*_2(F,\psi_0, Re, Ro).
    \label{eqn:StabilityBoundPsih}
  \end{equation}
  Using \eqref{eqn:StabilityBoundPsih}, estimate \eqref{eqn:YoungPhih} becomes
  \begin{equation}
    b(\psi^h; \eta, \Phi^h) \le \epsilon \|\Delta \Phi^h\|^2 +
      \frac{\Gamma_4^2}{4\epsilon} C^*_2(F,\psi_0,Re,Ro) \|\Delta \psi^h\|\,
      \|\Delta \eta\|^2.
    \label{eqn:bPsihbound}
  \end{equation}
  Combining \eqref{eqn:B2Inequality} and \eqref{eqn:bPsihbound} gives
  \begin{equation}
    \begin{split}
      \frac{1}{2} \frac{d}{dt} \|\nabla \Phi^h\|^2 + &\frac{1}{2}\left(2Re^{-1} -
        9 \epsilon \right)
        \|\Delta \Phi^h\|^2 \le \frac{1}{2 \epsilon}\left[\Gamma_0^2 \|\nabla \eta_t\|^2
        + \left( Re^{-2} + Ro^{-2} \Gamma_2^2 \right) \|\Delta \eta\|^2\right]\\
      & + \frac{\Gamma_1^2}{4 \epsilon} \|\Delta \psi\|^2 \|\Delta \eta\|^2
        + \frac{\Gamma_4}{4\epsilon}C^*_2(F,\psi_0,Re,Ro) \|\Delta \psi^h\|\,
        \|\Delta \eta\|^2 + C^*_1(\Gamma_3,\epsilon) \|\Delta \psi\|^4 \|\nabla \Phi^h\|^2.
    \end{split}
    \label{eqn:B3Inequality}
  \end{equation}
  Take $\epsilon = \nicefrac{Re^{-1}}{9}$ in \eqref{eqn:B3Inequality}.
  Letting $C^*_0(\Gamma_0) = \Gamma_0^2$,
  $C^*_3(F,\psi_0,Re,Ro,\Gamma_4) = \dfrac{\Gamma_4}{2}\, C^*_2(F,\psi_0,Re,Ro)$,
  $C^*_4(Re) = \frac{9}{2} Re$, $C^*_5(Re,\Gamma_3)=\frac{27}{256}\,9^3\,Re^{3}
  \Gamma_3^4$,
  $C^*_6(Re,Ro,\Gamma_2) = Re^{-2} + Ro^{-2}\Gamma_2^2$, and $C^*_7(\Gamma_1) =
  \dfrac{\Gamma_1^2}{2}$, \eqref{eqn:B3Inequality} reads
  \begin{equation}
    \begin{split}
      \frac{1}{2} \frac{d}{dt} &\|\nabla \Phi^h\|^2
        + \frac{Re^{-1}}{2} \|\Delta \Phi^h\|^2
        \le C^*_4(Re) \biggl[C^*_0(\Gamma_0) \|\nabla \eta_t\|^2
        + C^*_6(Re, Ro,\Gamma_2) \|\Delta \eta\|^2                            \\
      & + C^*_7(\Gamma_1)\, \|\Delta \psi\|^2 \|\Delta \eta\|^2
        + C^*_3(F,\psi_0,Re,Ro,\Gamma_4) \|\Delta \psi^h\|\,
        \|\Delta \eta\|^2\biggr]
        + C^*_5(Re,\Gamma_3) \|\Delta \psi\|^4 \|\nabla \Phi^h\|^2.
    \end{split}
    \label{eqn:NoEps}
  \end{equation}
  Let $a(t):= 2\,C^*_5(Re,\Gamma_3) \|\Delta \psi\|^4$ and
  \begin{equation}
    A(t) := \int_{0}^{t}\! a(t')\, dt' < \infty.
    \label{eqn:L4Bound}
  \end{equation}
  Multiplying \eqref{eqn:NoEps} by the integrating factor $e^{-A(t)}$, we get
  \begin{align*}
    \biggl\{ \frac{d}{dt}\left[\|\nabla \Phi^h\|^2\right]
      &- 2\, C^*_5(Re,\Gamma_3) \|\Delta \psi\|^4 \|\nabla \Phi^h\|^2\biggr\} e^{-A(t)}
        + Re^{-1} \|\Delta \Phi^h\|^2 e^{-A(t)}                               \\
      & \le 2\, C^*_4(Re) \biggl[C^*_0(\Gamma_0) \|\nabla \eta_t\|^2
        + C^*_6(Re,Ro,\Gamma_2) \|\Delta \eta\|^2 + C^*_7(\Gamma_1)\,
        \|\Delta \psi\|^2 \|\Delta \eta\|^2                                   \\
      & \qquad+ C^*_3(F,\psi_0,Re,Ro,\Gamma_4) \|\Delta \psi^h\|\, \|\Delta
        \eta\|^2\biggr] e^{-A(t)},
  \end{align*}
  which can also be written as
  \begin{align*}
    \biggl\{ e^{-A(t)}\frac{d}{dt}
      & \left[\|\nabla \Phi^h\|^2\right]
      - \frac{d}{dt}\bigl[ A(t)\bigr] e^{-A(t)} \|\nabla \Phi^h\|^2\biggr\}
      + Re^{-1} \|\Delta \Phi^h\|^2 e^{-A(t)}                                 \\
    & \le 2\,C^*_4(Re) \biggl[C^*_0(\Gamma_0) \|\nabla \eta_t\|^2
      + C^*_6(Re,Ro,\Gamma_2) \|\Delta \eta\|^2 + C^*_7(\Gamma_1)\,
      \|\Delta \psi\|^2 \|\Delta \eta\|^2                                     \\
    &\qquad + C^*_3(F,\psi_0,Re,Ro,\Gamma_4) \|\Delta \psi^h\|\, \|\Delta
      \eta\|^2\biggr] e^{-A(t)},
  \end{align*}
  and simplifies to
  \begin{equation}
    \label{eqn:SimplifedIntegratingFactor}
    \begin{aligned}
      \frac{d}{dt}\bigl[e^{-A(t)} &\|\nabla \Phi^h\|^2\bigr]
        + Re^{-1} \|\Delta \Phi^h\|^2 e^{-A(t)}                               \\
      & \le 2\, C^*_4(Re) \biggl[C^*_0(\Gamma_0) \|\nabla \eta_t\|^2
        + C^*_6(Re,Ro,\Gamma_2) \|\Delta \eta\|^2 + C^*_7(\Gamma_1)\,
        \|\Delta \psi\|^2 \|\Delta \eta\|^2                                   \\
      &\qquad + C^*_3(F,\psi_0,Re,Ro,\Gamma_4)\, \|\Delta \psi^h\|\,
        \|\Delta \eta\|^2\biggr] e^{-A(t)}.
    \end{aligned}
  \end{equation}
  Now, integrating \eqref{eqn:SimplifedIntegratingFactor} over $[0,T]$ and
  multiplying by $e^{A(T)}$ gives
  \begin{equation}
    \begin{split}
      \|\nabla \Phi^h(T)\|^2 + Re^{-1} \int_0^T\! &\|\Delta \Phi^h\|^2
        e^{A(T) - A(t)}\, dt \le e^{A(T) - A(0)} \|\nabla \Phi^h(0)\|^2       \\
      & + 2\, C^*_4(Re)\biggl[ \int_0^T\!
          C^*_0(\Gamma_0) \|\nabla \eta_t\|^2
        + C^*_6(Re,Ro,\Gamma_2) \|\Delta \eta\|^2 e^{A(T) - A(t)}\, dt        \\
      & + \int_0^T\!  \left( C^*_7(\Gamma_1)\, \|\Delta \psi\|^2
        +  C^*_3(F,\psi_0,Re,Ro,\Gamma_4)\,\|\Delta \psi^h\|\right)
        \|\Delta \eta\|^2 e^{A(T) - A(t)}\, dt\biggr].
    \end{split}
    \label{eqn:IntegratedInequality}
  \end{equation}
  Noting that $e^{A(T) - A(t)} \ge 1$, $e^{A(T) - A(t)} \le e^{A(T)}$, and
  $A(0) = 0$, \eqref{eqn:IntegratedInequality} implies
  \begin{equation}
    \begin{split}
      \|\nabla \Phi^h(T)\|^2 + Re^{-1} \int_0^T\! \|\Delta \Phi^h\|^2 &\, dt
        \le C^*_8(T,Re,\Gamma_3) \|\nabla \Phi^h(0)\|^2                       \\
      & + C^*_9(T,Re,\Gamma_3)\biggl[ \int_0^T\!
          C^*_0(\Gamma_0) \|\nabla \eta_t\|^2
        + C^*_6(Re,Ro,\Gamma_2) \|\Delta \eta\|^2\, dt                        \\
      & + \int_0^T\!  \left(C^*_7(\Gamma_1)\, \|\Delta \psi\|^2
        + C^*_3(F,\psi_0,Re,Ro,\Gamma_4)\,\|\Delta \psi^h\|\right)
        \|\Delta \eta\|^2\, dt\biggr],
    \end{split} \label{eqn:CTREInequality}
  \end{equation}
  where
  \begin{align}
    C^*_8(T,Re,\Gamma_3) &= \exp\!\left(2\,\dfrac{27}{256}\, 9^3\, Re^3\,
      \Gamma_3^4\, \int_{0}^{T}\!\|\Delta \psi\|^4\, dt\right), \label{eqn:C1TRe}\\
    C^*_9(T,Re,\Gamma_3) &= 9 Re\, \exp\!\left(2\,\dfrac{27}{256}\, 9^3\, Re^3\,
      \Gamma_3^4\, \int_{0}^{T}\!\|\Delta \psi\|^4\, dt\right). \label{eqn:C2TRe}
  \end{align}
  By the Cauchy-Schwarz inequality we have
  \begin{align}
    \int_0^T\! \|\Delta \psi^h\| \|\Delta \eta\|^2\, dt &\le
      \|\Delta \psi^h\|_{L^2(0,T;L^2)} \|\Delta \eta\|^2_{L^4(0,T;L^2)},
    \label{eqn:HolderPsih}                                                    \\
    \int_0^T\! \|\Delta \psi\|^2 \|\Delta \eta\|^2\, dt &\le
      \|\Delta \psi\|^2_{L^4(0,T;L^2)} \|\Delta \eta\|^2_{L^4(0,T;L^2)}.
    \label{eqn:HolderPsi}
  \end{align}
  Note that $\|\Delta \psi^h\|_{L^2(0,T;L^2)}\le C^*_{10}(Re,Ro,F,\psi_0)$ from
  the stability bound \eqref{eqn:Stability} and (by hypothesis) $\|\Delta
  \psi\|_{L^4(0,T;L^2)} \le C^*_{11}$. Thus, \eqref{eqn:CTREInequality} can be
  written as
  \begin{equation}
    \label{eqn:ApplyL4StabilityBound}
    \begin{aligned}
      \|\nabla \Phi^h(T)\|^2 &+ Re^{-1} \int_0^T\! \|\Delta \Phi^h\|^2\, dt
        \le C^*_8(T,Re,\Gamma_3) \|\nabla \Phi^h(0)\|^2                       \\
      & + C^*_9(T,Re,\Gamma_3)\biggl[ \int_0^T\!
          C^*_0(\Gamma_0) \|\nabla \eta_t\|^2
        + C^*_6(Re,Ro,\Gamma_2) \|\Delta \eta\|^2\, dt                        \\
      & + \bigg(C^*_7(\Gamma_1)\,C^*_{11}
        + C^*_3(F,\psi_0,Re,Ro,\Gamma_4)\,C^*_{10}(Re,Ro,F,\psi_0)\bigg)
        \|\Delta \eta\|^2_{L^4(0,T;L^2)}\biggr].
    \end{aligned}
  \end{equation}
  \begin{remark}
    We note that the stability bound in \autoref{prop:Stability} does not
    provide an estimate for $\|\Delta \psi^h\|_{L^4(0,T;L^2)}$, and thus was the
    reasoning for treating the nonlinear terms $b(\eta;\psi,\Phi^h)$ and
    $b(\psi^h;\eta,\Phi^h)$ in \eqref{eqn:HolderError} differently.
  \end{remark}
  {\setlength{\parindent}{0cm}
  Adding $\|\nabla \eta(T)\|^2 + Re^{-1} \int_0^T\! \|\Delta \eta\|^2\, dt$
  to both sides of \eqref{eqn:ApplyL4StabilityBound} and using the triangle
  inequality gives}
  \begin{equation}
    \begin{split}
      \frac{1}{2} \|\nabla ( \psi &- \psi^h)(T) \|^2
      + \frac{Re^{-1}}{2} \int_0^T\! \|\Delta \left(\psi - \psi^h\right)\|^2\, dt
        \le C^*_8(T,Re,\Gamma_3) \|\nabla \Phi^h(0)\|^2                       \\
      & + C^*_9(T,Re,\Gamma_3) \int_0^T\!
        C^*_0(\Gamma_0) \|\nabla \left(\psi - \lambda^h\right)_t\|^2
      + (Re^{-1 + } C^*_6(Re,Ro,\Gamma_2))
      \|\Delta \left(\psi - \lambda^h\right)\|^2\, dt                         \\
      & + \bigl[C^*_7(\Gamma_1)\, C^*_{11}                                    \\
      & + C^*_3(F,\psi_0,Re,Ro,\Gamma_4)\,C^*_{10}(Re,Ro,F, \psi_0)\bigr]
       \|\Delta \left(\psi - \lambda^h\right)\|^2_{L^4(0,T;L^2)}
        + \|\nabla \left(\psi - \lambda^h\right)(T)\|^2.
    \end{split}
    \label{eqn:TriangleIneq1}
  \end{equation}
  Since \(\|\Phi^h(0)\| \leq \|e(0)\| + \|\eta(0)\|\), inequality
  \eqref{eqn:TriangleIneq1} yields
  \begin{equation}
    \begin{split}
      \frac{1}{2} \|\nabla ( \psi &- \psi^h)(T) \|^2
      + \frac{Re^{-1}}{2} \int_0^T\! \|\Delta \left(\psi - \psi^h\right)\|^2\, dt
        \le \,C^*_8(T,Re,\Gamma_3) \bigg(\|\nabla e(0)\|^2 + \|\nabla (\psi -
        \lambda^h)(0)\|^2\bigg)                                               \\
      & + C^*_9(T,Re,\Gamma_3) \int_0^T\!
          C^*_0(\Gamma_0) \|\nabla \left( \psi - \lambda^h\right)_t\|^2
        + \bigl(Re^{-1} + C^*_6(Re,Ro,\Gamma_2)\bigr)
        \|\Delta \left(\psi - \lambda^h\right)\|^2\, dt                       \\
      & + \bigl[C^*_7(\Gamma_1)\, C^*_{11}
        + C^*_3(F,\psi_0,Re,Ro,\Gamma_4)\,C^*_{10}(Re,Ro,F, \psi_0)\bigr]
        \|\Delta \left(\psi - \lambda^h\right)\|^2_{L^4(0,T;L^2)}             \\
      & + \|\nabla \left(\psi - \lambda^h\right)(T)\|^2.
    \end{split}
    \label{eqn:TriangleIneq2}
  \end{equation}
  Finally, taking $\inf_{\lambda^h : [0, T] \to X^h}$ of both sides of
  \eqref{eqn:TriangleIneq2} and letting
  \begin{align*}
    C = \max\biggl\{&2 C^*_8(T,Re,\Gamma_3), 2 C^*_9(T,Re,\Gamma_3) \max\{1,
    Re^{-1} + C^*_6(Re, Ro, \Gamma_2)\},                                      \\
      & 2 [C^*_7(\Gamma_1)\, C^*_{11}
      + C^*_3(F,\psi_0,Re,Ro,\Gamma_4) C^*_{10}(F, Re, Ro, \psi_0)],
      2 \biggr\}
  \end{align*}
  gives
  \begin{equation*}
    \begin{aligned}
      \|\nabla \big(\psi &- \psi^h\big)(T) \|^2
        + Re^{-1} \int_0^T\! \|\Delta \left(\psi - \psi^h\right)\|^2\, dt
        \le C\,\biggl\{\bigl\|\nabla\left[\psi - \psi^h\right](0)\bigr\|^2    \\
          &+ \inf_{\lambda^h : [0,T] \to  X^h} \biggl[
      \bigl\|\nabla\left[\psi - \lambda^h\right](0)\bigr\|^2
       + \int_0^T\!
         \left\|\nabla \left(\psi - \lambda^h\right)_t\right\|^2
       + \left\|\Delta \left(\psi - \lambda^h\right)\right\|^2\, dt           \\
      &+ \left\|\Delta \left(\psi - \lambda^h\right)\right\|^2_{L^4(0,T;L^2)}
       + \|\nabla \left(\psi - \lambda^h\right)(T)\|^2\biggr]\biggr\},
    \end{aligned}
  \end{equation*}
  which is the desired result.
\end{proof}

Next we determine the FE convergence rates yielded by the error estimate
\eqref{eqn:SemiConvergence} in \autoref{thm:SemiConvergence} for the Argyris
element. To this end, in the remainder of this section we let \(X^h \subseteq
X\) denote the FE space associated with the Argyris element. Furthermore, we
assume the nodes of the FE mesh do not move. Finally, let \(I^h\) be the
\(\mathbb{P}^5\)-interpolation operator associated with the Argyris element (see
Theorem 6.1.1 in \cite{Ciarlet}). The following two lemmas will be used in
\autoref{thm:SemiInterp} to determine the FE convergence rates for the Argyris
element.
\begin{lemma} \label{lma:Interpolation}
  Assuming that $\psi, \psi_t \in H^6$, we have that
  \begin{align}
    \frac{\partial}{\partial t}\left( I^h \psi \right) &=
      I^h \left(\frac{\partial \psi}{\partial t}\right), \text{ and }
      \label{eqn:dInterp}                                                     \\
    \left\|\frac{\partial \left[\nabla \psi\right]}{\partial t}
      - \frac{\partial}{\partial t} \nabla \left(I^h \psi\right)\right\|
      &\le C\, h^5\, \left|\frac{\partial \psi}{\partial t}\right|_6.
      \label{eqn:H1InterpolationError}
  \end{align}
  \begin{remark} \label{rmk:H6imbedC1}
    Estimate (32) in Theorem 6 in Section 5.6 of \cite{Evans2010}
    shows that $H^6 \hookrightarrow C^1$. Thus, the interpolation operator $I^h$
    can be applied to $\psi$ and $\psi_t$.
  \end{remark}
\end{lemma}
\begin{proof}
  Estimate \eqref{eqn:dInterp} follows from the explicit formulas for the
  $\mathbb{P}^5$ interpolant, $I^h$ (see \cite{Ciarlet}). Estimate
  \eqref{eqn:H1InterpolationError} follows from a combination of
  \eqref{eqn:dInterp} and
  estimate (6.1.9) in Theorem 6.1.1 from \cite{Ciarlet} with $p=q=2$ and $m=1$.
\end{proof}

\begin{lemma} \label{lma:IntegralErrors}
  Suppose that $\psi, \psi_t \in H^6(\Omega)$. Then
  \begin{equation}
    \int_{0}^{T}\! \|\nabla \left(\psi - I^h \psi\right)_t\|^2 + \|\Delta
    \left(\psi - I^h \psi\right) \|^2\, dt \le C\, h^8
    \int_{0}^{T}\! h^2 \,|\psi_t|_6^2 + |\psi|_6^2\, dt
    \label{eqn:IntegralErrors}
  \end{equation}
  and
  \begin{equation}
    \|\Delta\left(\psi- I^h \psi\right)\|_{L^4(0,T; L^2(\Omega))}^2 \le C h^8
      |\psi|_{L^4(0,T;H^6(\Omega))}^2. \label{eqn:L4Interpolation}
  \end{equation}
\end{lemma}
\begin{proof}
  At each time instance we see from inequality (6.1.9) in \cite{Ciarlet} that
  $\|\Delta\left(\psi - I^h \psi\right)\| \le C\, h^4\, |\psi|_6$. Squaring and
  integrating this and using the interpolation error bound
  \eqref{eqn:H1InterpolationError} from \autoref{lma:Interpolation} gives the
  first estimate. The second estimate follows analogously, \ie
  \begin{equation}
    \|\Delta \left(\psi - I^h \psi\right)\|_{L^4(0,T;L^2(\Omega))}
    = \left(\int_{0}^{T}\! \|\Delta\left(\psi - I^h \psi\right)\|^4\, dt
      \right)^{\frac{1}{4}}
    \le C\, h^4 \left(\int_{0}^{T}\! |\psi|_6^4\, dt\right)^{\frac{1}{4}},
  \end{equation}
  which proves \eqref{eqn:L4Interpolation}.
\end{proof}

\begin{thm} \label{thm:SemiInterp}
  Suppose that $\psi, \psi_t \in H^6(\Omega)$. Suppose also the
  assumptions of \autoref{thm:SemiConvergence} hold. Then
  \begin{equation}
    \begin{aligned}
      \|\nabla \big(\psi &- \psi^h\big)(T) \|^2
        + Re^{-1} \int_0^T\! \|\Delta \left(\psi - \psi^h\right)\|^2\, dt     \\
      & \le h^8\, C\, \biggl\{ h^2\, |\psi|_6^2
        + h^2\,\|\psi_t\|_{L^2(0,T;H^6(\Omega))}^2
        + \|\psi\|_{L^2(0,T;H^6(\Omega))}^2
        + \|\psi\|_{L^4(0,T;H^6(\Omega))}^2 \biggr\}.
    \end{aligned}
    \label{eqn:SemiInterp}
  \end{equation}
\end{thm}
\begin{proof}
  The proof follows from \autoref{thm:SemiConvergence},
  \autoref{lma:Interpolation}, and \autoref{lma:IntegralErrors}.
\end{proof}

  \section{Numerical Results} \label{sec:Results}
  In this section we verify the theoretical error estimates developed in
\autoref{sec:Errors}.
As noted in Section 6.1 of \cite{Ciarlet} (see also Section 13.2 in
\cite{Gunzburger89}, Section 3.1 in \cite{Johnson}, and Theorem 5.2 in
\cite{Braess}), in order to develop a conforming FE discretization for the QGE
\eqref{eqn:QGEStrong}, we are faced with the problem of constructing FE
subspaces of $H^2_0(\Omega)$. Since the standard, piecewise polynomial FE spaces
are locally regular, this construction amounts in practice to finding FE spaces
$X^h$ that satisfy the inclusion $X^h \subset C^1({\bar \Omega})$, \ie $C^1$
FEs. Several FEs meet this requirement (\eg Section 6.1 in
\cite{Ciarlet}, Section 13.2 in \cite{Gunzburger89}, and Section 2.5 in
\cite{Braess}): the Argyris triangular element, the Bell triangular element, the
Hsieh-Clough-Tocher triangular element (a macroelement), and the
Bogner-Fox-Schmit rectangular element. In our numerical investigation, we will
use the Argyris triangular element, depicted in
\autoref{fig:Argyris}. Additionally, we note that
\eqref{eqn:SemiDiscretization}-\eqref{eqn:initialConditionProjection} is only a
semi-discretization, since the formulation is still continuous in time, but
discretized in space. For this numerical discretization, we apply the method of
lines in the time domain, \ie we use a finite difference approximation (implicit
Euler scheme) for the time derivative.
\begin{figure}[h]
	\centering
\begin{tikzpicture}[scale=0.5]
	\path[coordinate] (0,0) coordinate(A)
		++(60:3cm) coordinate(D)
		++(60:3cm) coordinate(B)
		++(-60:3cm) coordinate(E)
		++(-60:3cm) coordinate(C)
		++(180:3cm) coordinate(F);
	\draw (A) -- (D) -- (B) -- (E) -- (C) -- (F) -- cycle;
	\filldraw[black] (A) circle(3pt); 
	\filldraw[black] (B) circle(3pt); 
	\filldraw[black] (C) circle(3pt); 
	\draw[black] (A) circle(6pt); 
	\draw[black] (B) circle(6pt); 
	\draw[black] (C) circle(6pt); 
	\draw[black] (A) circle(9pt); 
	\draw[black] (B) circle(9pt); 
	\draw[black] (C) circle(9pt); 
	\draw (1.067,2.848) -- (D) -- (1.933,2.348); 
	\draw (4.067,2.348) -- (E) -- (4.933,2.848); 
	\draw (3,-0.5cm) -- (F) -- (3,0.5cm); 
\end{tikzpicture}
	\caption{Argyris element with its 21 degrees of freedom.}
	\label{fig:Argyris}
\end{figure}
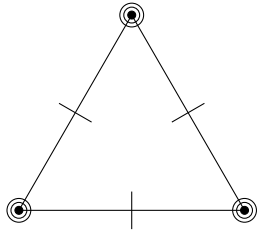

We apply Newton's method to solve the resulting nonlinear system at each time
step.
We test for convergence of the nonlinear solver by examining the $\ell^2$-norm of the
Newton update; when the norm of the update is less than $10^{-8}$, then we
consider the iteration to have converged.

We use $Re = 1$ and $Ro = 1$ in all of the following computational tests. The
variables $k$ and $h$ respectively refer to the time and spatial discretization
stepsizes.

\subsection*{Test 1}
We use an exact solution
\begin{equation}
  \psi(t;x,y) = \left[\sin(\pi x) \sin(\pi y)\right]^2 \sin(t)
  \label{eqn:Test1}
\end{equation}
with spatial domain $\Omega = [0,1]^2$.
This is similar to Test 3 in \cite{Foster}. The considered time interval is
$\left[0,\frac{\pi}{2}\right]$. The forcing term $F$ is derived by the method of
manufactured solutions. The results of this experiment are summarized in
\autoref{tab:Test1Space}, which displays the orders of convergence of the FE
discretization in \(L^2\), \(H^1\), and \(H^2\) norms for differing \(h\). The
results in \autoref{tab:Test1Space} are plotted in \autoref{fig:Test1Space}. Note that
the observed orders of convergence are close to the theoretical error estimates
developed in \autoref{sec:Errors}. The $L^2$ order, however, drops off for the
last spatial discretization due the error per node being near machine precision.

\begin{table}[h!]
    \centering
    \begin{tabular}{|c|c|c|c|c|c|c|c|c|}
      \hline
      $k$ & $h$ & DoFs & $e_{L^2}$ & $L^2$ order & $e_{H^1}$ & $H^1$ order & $e_{H^2}$ & $H^2$ order\\
      \hline
      $\nicefrac{1}{8192}$ & $\nicefrac{1}{2}$ & $38$ & $1.23\times 10^{-2}$ & $-$ & $1.18\times 10^{-1}$ & $-$ & $1.57\times 10^0$ & $-$ \\ [0.2em]
      $\nicefrac{1}{8192}$ & $\nicefrac{1}{4}$ & $174$ & $2.12\times 10^{-5}$ & $9.18$ & $7.31\times 10^{-4}$ & $7.34$ & $2.79\times 10^{-2}$ & $5.81$ \\ [0.2em]
      $\nicefrac{1}{8192}$ & $\nicefrac{1}{8}$ & $662$ & $7.88\times 10^{-7}$ & $4.75$ & $4.59\times 10^{-5}$ & $3.99$ & $3.04\times 10^{-3}$ & $3.20$ \\ [0.2em]
      $\nicefrac{1}{8192}$ & $\nicefrac{1}{16}$ & $2853$ & $7.87\times 10^{-9}$ & $6.65$ & $9.05\times 10^{-7}$ & $5.67$ & $1.29\times 10^{-4}$ & $4.56$ \\ [0.2em]
      $\nicefrac{1}{8192}$ & $\nicefrac{1}{32}$ & $11690$ & $6.97\times 10^{-11}$ & $6.82$ & $1.88\times 10^{-8}$ & $5.59$ & $5.98\times 10^{-6}$ & $4.43$ \\ [0.2em]
      $\nicefrac{1}{8192}$ & $\nicefrac{1}{64}$ & $47958$ & $7.23\times 10^{-12}$ & $3.27$ & $5.26\times 10^{-10}$ & $5.16$ & $3.43\times 10^{-7}$ & $4.12$ \\ [0.2em]
      \hline
    \end{tabular}

  \caption{Test 1: spatial orders of convergence with exact solution
  \eqref{eqn:Test1}.}
  \label{tab:Test1Space}
\end{table}

\begin{figure}
  \centering
    \includegraphics[scale=0.5]{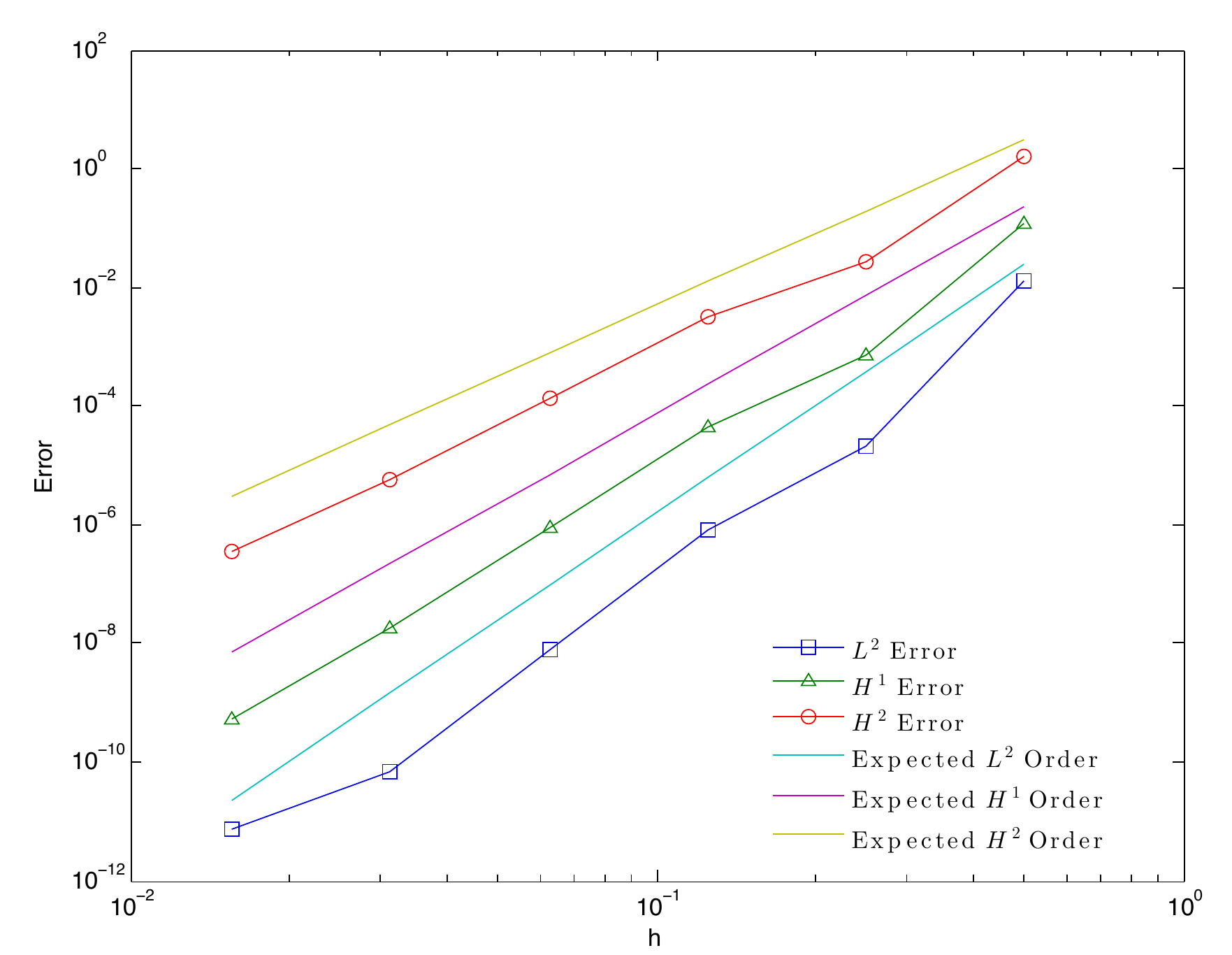}
    \caption{Test 1: orders of convergence in space for the full discretization
    of \eqref{eqn:QGE_psi} with exact solution \eqref{eqn:Test1}.}
  \label{fig:Test1Space}
\end{figure}

\subsection*{Test 2}
For this test we take the exact solution to be
\begin{equation}
  \psi(t;x,y) = \left[\left(1-\frac{x}{3}\right)\left(1-e^{-20\,x}\right)
  \sin(\pi y)\right]^2 \sin(t)
  \label{eqn:Test2}
\end{equation}
with spatial domain $\Omega = [0,3] \times [0,1]$, which corresponds to
Test 6 in \cite{Foster} with a time-dependent term.
The time interval for integration is $[0,0.5]$. A boundary layer will form along
the western edge of the problem domain in this example. Note that the observed
orders of convergence match the theoretical error estimates developed in
\autoref{sec:Errors}. The results in \autoref{tab:Test2Space} are also plotted
in \autoref{fig:Test2Space}.

\begin{table}
  \centering
  \begin{tabular}{|c|c|c|c|c|c|c|c|c|}
    \hline
    $k$ & $h$ & DoFs & $e_{L^2}$ & $L^2$ order & $e_{H^1}$ & $H^1$ order & $e_{H^2}$ & $H^2$ order\\
    \hline
    $\nicefrac{1}{8192}$ & $\nicefrac{1}{2}$ & $38$ & $2.86\times 10^{-2}$ & $-$ & $5.16\times 10^{-1}$ & $-$ & $1.82\times 10^{1}$ & $-$\\
    $\nicefrac{1}{8192}$ & $\nicefrac{1}{4}$ & $174$ & $4.79\times 10^{-3}$ & $2.58$ & $1.75\times 10^{-1}$ & $1.56$ & $9.28\times 10^0$ & $0.973$\\
    $\nicefrac{1}{8192}$ & $\nicefrac{1}{8}$ & $662$ & $5.04\times 10^{-4}$ & $3.25$ & $3.38\times 10^{-2}$ & $2.37$ & $2.96\times 10^0$ & $1.65$\\
    $\nicefrac{1}{8192}$ & $\nicefrac{1}{16}$ & $2853$ & $1.65\times 10^{-5}$ & $4.94$ & $2.17\times 10^{-3}$ & $3.96$ & $3.67\times 10^{-1}$ & $3.01$\\
    $\nicefrac{1}{8192}$ & $\nicefrac{1}{32}$ & $11690$ & $4.17\times 10^{-7}$ & $5.30$ & $1.07\times 10^{-4}$ & $4.34$ & $3.47\times 10^{-2}$ & $3.40$\\
    $\nicefrac{1}{8192}$ &  $\nicefrac{1}{64}$ & $47958$ & $7.28\times 10^{-9}$ & $5.84$ & $3.70\times 10^{-6}$ & $4.86$ & $2.37\times 10^{-3}$ & $3.87$\\
    \hline
  \end{tabular}
  \caption{Test 2: spatial orders of convergence with exact solution
  \eqref{eqn:Test2}.}
    \label{tab:Test2Space}
\end{table}

\begin{figure}
    \centering
    \includegraphics[scale=0.5]{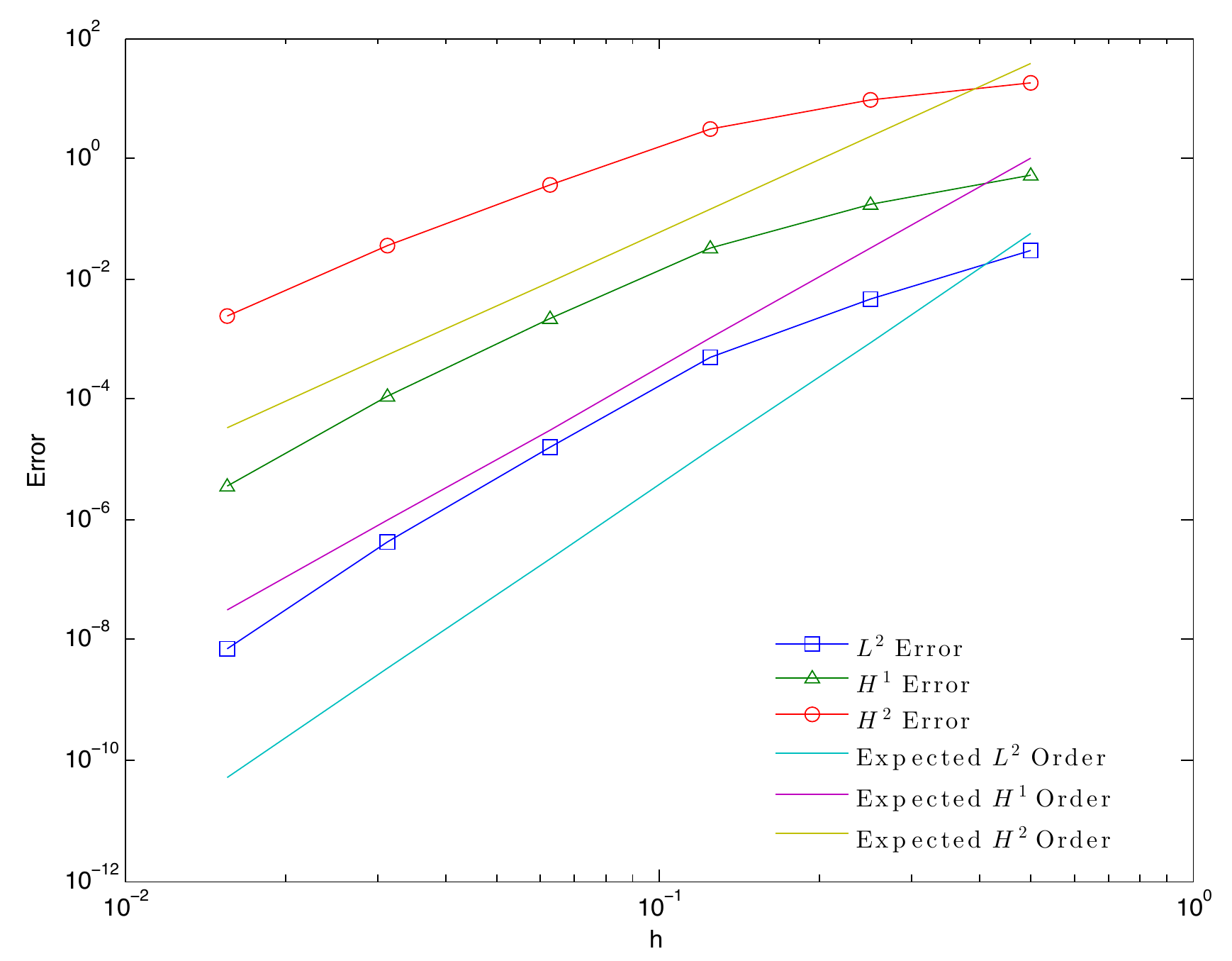}
    \caption{Test 2: orders of convergence in space for the full discretization
    of \eqref{eqn:QGE_psi} with exact solution \eqref{eqn:Test2}.}
  \label{fig:Test2Space}
\end{figure}

  \section{Conclusions} \label{sec:Conclusions}
  In this paper we studied the conforming FE semi-discretization of the pure
streamfunction form of the QGE\@. This semi-discretization requires a $C^1$ FE,
for which we chose the Argyris element. In \autoref{sec:Errors} we developed
rigorous error estimates for the conforming FE semi-discretization of the QGE\@.
For this analysis only the fact that the semi-discretization is conforming was
used. We showed that the orders of convergence are optimal.

Numerical experiments for the QGE (\autoref{sec:Results}) with the Argyris
element, were also carried out. The code which was developed and verified for
the stationary QGE in \cite{Foster} was then modified to deal with
time-dependence. We applied an implicit Euler scheme and verified numerically
the theoretical spatial rates of convergence developed for the
semi-discretization.

The QGE have many unique challenges for numerical modeling. These challenges
include (but are not limited to) unstable solutions, resulting from internal
layers and western boundary layers, and high computational cost for large
domains, such as the North Atlantic. To address these issues we plan to extend
these studies in several directions to include stabilization methods. We are
also interested in incorporating empirical wind-stress data, which will require
parameter estimation techniques.

  \bibliographystyle{siam}
  \bibliography{QGE}

\begin{thebibliography}{10}

\bibitem{Adcroft98}
{\sc A.~Adcroft and D.~Marshall}, {\em How slippery are piecewise-constant
  coastlines in numerical ocean models?}, Tellus, Ser. A and Ser. B-Dyn.
  Meteorol. Oceanogr., 50 (1998,), pp.~95--108.

\bibitem{Braess}
{\sc D.~Braess}, {\em Finite elements: {T}heory, fast solvers, and applications
  in solid mechanics}, Cambridge University Press, 2001.

\bibitem{Cayco86}
{\sc M.~E. Cayco and R.~A. Nicolaides}, {\em Finite element technique for
  optimal pressure recovery from stream function formulation of viscous flows},
  Math. of Comp., 46 (1986).

\bibitem{Ciarlet}
{\sc P.~G. Ciarlet}, {\em The finite element method for elliptic problems},
  North-Holland, 1978.

\bibitem{Cushman94}
{\sc B.~Cushman-Roisin}, {\em Introduction to geophysical fluid dynamics},
  Prentice Hall, Englewood Cliffs, New Jersey, 1994.

\bibitem{Cushman11}
{\sc B.~Cushman-Roisin and J.~M. Beckers}, {\em Introduction to geophysical
  fluid dynamics: Physical and numerical aspects}, International Geophysics,
  Elsevier Science \& Technology, 2011.

\bibitem{Dijkstra05}
{\sc H.~E. Dijkstra}, {\em Nonlinear physical oceanography: {A} dynamical
  systems approach to the large scale ocean circulation and el Nino}, vol.~28,
  Springer Verlag, 2005.

\bibitem{Dupont03}
{\sc F.~Dupont, D.~N. Straub, and C.~A. Lin}, {\em Influence of a step-like
  coastline on the basin scale vorticity budget of mid-latitude gyre models},
  Tellus, Ser. A-Dyn Meteorol. Oceanogr., 55 (2003), pp.~255--272.

\bibitem{Evans2010}
{\sc L.~C. Evans}, {\em Partial Differential Equations}, vol.~19, American
  Mathematical Society, Providence, 2010.

\bibitem{Fairag98}
{\sc F.~Fairag}, {\em A two-level finite-element discretization of the stream
  function form of the {N}avier-{S}tokes equations}, Comp. Math. Applic., 36
  (1998), pp.~117--127.

\bibitem{FosterThesis}
{\sc E.~L. Foster}, {\em Finite Elements for the quasi-geostrophic equations of
  the ocean}, PhD thesis, Virginia Polytechnic Institute and State University,
  2013.

\bibitem{Foster}
{\sc E.~L. Foster, T.~Iliescu, and Z.~Wang}, {\em A finite element
  discretization of the streamfunction formulation of the stationary
  quasi-geostrophic equations of the ocean}, Comp. Meth. Appl. Mech. and Eng.,
  261-262 (2013), pp.~105--117.

\bibitem{Greatbatch00}
{\sc R.~J. Greatbatch and B.~T. Nadiga}, {\em Four-gyre circulation in a
  barotropic model with double-gyre wind forcing}, J. Phys. Oceanogr., 30
  (2000), pp.~1461--1471.

\bibitem{Gunzburger89}
{\sc M.~D. Gunzburger}, {\em Finite element methods for viscous incompressible
  flows}, Computer Science and Scientific Computing, Academic Press Inc, 1989.
\newblock A Guide to Theory, Practice, and Algorithms.

\bibitem{Johnson}
{\sc C.~Johnson}, {\em Numerical solution of partial differential equations by
  the finite element method}, vol.~32, Cambridge University Press, New York,
  1987.

\bibitem{Layton08}
{\sc W.~J. Layton}, {\em Introduction to the numerical analysis of
  incompressible viscous flows}, vol.~6, Society for Industrial and Applied
  Mathematics (SIAM), 2008.

\bibitem{Majda}
{\sc A.~J. Majda and X.~Wang}, {\em Non-linear dynamics and statistical
  theories for basic geophysical flows}, Cambridge University Press, 2006.

\bibitem{Pedlosky92}
{\sc J.~Pedlosky}, {\em Geophysical fluid dynamics}, Springer, second~ed.,
  1992.

\bibitem{San11}
{\sc O.~San, A.~E. Staples, and T.~Iliescu}, {\em Approximate deconvolution
  large eddy simulation of a barotropic ocean circulation model}, Ocean Model.,
  40 (2011), pp.~120--132.

\bibitem{Vallis06}
{\sc G.~K. Vallis}, {\em Atmosphere and ocean fluid dynamics: {F}undamentals
  and large-scale circulation}, Cambridge University Press, 2006.

\bibitem{Wang94}
{\sc J.~Wang and G.~K. Vallis}, {\em Emergence of {F}ofonoff states in inviscid
  and viscous ocean circulation models}, J. Mar. Res., 52 (1994), pp.~83--127.

\bibitem{Wang08}
{\sc Q.~Wang, S.~Danilov, and J.~Schr\"oter}, {\em Finite element ocean
  circulation model based on triangular prismatic elements, with application in
  studying the effect of topography representation}, J. of Geophys. Res., 113
  (2008).

\end{thebibliography}
\end{document}